\documentclass[11pt,twoside]{article}
\usepackage{mathrsfs}
\usepackage{amsmath}
\usepackage{amsthm}
\usepackage{amsfonts}
\usepackage{amssymb}
\usepackage{cite}
\usepackage{color}
\usepackage{latexsym}
\usepackage[all]{xy}

\date{\empty}
\allowdisplaybreaks[4] 

\numberwithin{equation}{section} \theoremstyle{plain}
\newtheorem*{thm*}{Main Theorem}
\newtheorem{theorem}{Theorem}[section]
\newtheorem{corollary}[theorem]{Corollary}
\newtheorem*{corollary*}{Corollary}

\newtheorem*{claim*}{Claim}
\newtheorem{lemma}[theorem]{Lemma}
\newtheorem*{lemma*}{Lemma}
\newtheorem{proposition}[theorem]{Proposition}
\newtheorem*{proposition*}{Proposition}

\newtheorem*{remark*}{Remark}
\newtheorem{example}[theorem]{Example}
\newtheorem*{example*}{Example}

\newtheorem*{question*}{Question}
\newtheorem{definition}[theorem]{Definition}
\newtheorem*{definition*}{Definition}

\newtheorem*{acknowledgements*}{ACKNOWLEDGEMENTS}

\begin{document}
\begin{center}
{\large \bf  $m$-weak group inverses in a ring with involution }

\vspace{0.4cm} {\small \bf Yukun Zhou}
\footnote{ Yukun Zhou ( E-mail:2516856280@qq.com): School of Mathematics, Southeast University, Nanjing 210096, China.}
\vspace{0.4cm} {\small \bf, Jianlong Chen}
\footnote{ Jianlong Chen (Corresponding author E-mail: jlchen@seu.edu.cn): School of Mathematics, Southeast University, Nanjing 210096, China.}
\vspace{0.4cm} {\small \bf, Mengmeng Zhou}
\footnote{Mengmeng Zhou (E-mail: mmz9202@163.com): School of Mathematics, Southeast University, Nanjing 210096, China.}

\end{center}
\bigskip

{ \bf  Abstract:}  \leftskip0truemm\rightskip0truemm
 In a unitary ring with involution, we prove that each element has at most one weak group inverse if and only if each idempotent element has a unique weak group inverse. Furthermore, we define the $m$-weak group inverse and show some properties of $m$-weak group inverse.


 \bigskip

\section {\bf Introduction}
In 1958, Drazin \cite{D} introduced the definition of the pseudo inverse in rings and semigroups, which is called the Drazin inverse later. It's well known that an element is Drazin invertible if and only if this element is strongly $\pi$-regular. In 2014, Manjunatha Prasad et al. \cite{MPM} introduced the core-EP inverse of complex matrices. In 2017, Gao et al. \cite{GC} generalized the core-EP inverse to rings with involution, where an involution $a\mapsto a^{*}$ is an anti-isomorphism such that
$(a+b)^{*}=a^{*}+b^{*}, (ab)^{*}=b^{*}a^{*}, (a^{*})^{*}=a$ for any $a, b\in R.$ In the rest of this paper, we restrict $R$ is a unitary ring with involution.

In 2018, Wang et al. \cite{WC} defined weak group inverses of complex matrices. Then, Zhou et al. \cite{Z} generalized the weak group inverses to proper $\ast$-rings and proved that each element in a proper $\ast$-ring has at most one weak group inverse. They gave an example to show that the weak group inverse of an element may not be unique in $R$. In addition, they proved that an element in a proper $\ast$-ring is weak group invertible if and only if this element has group-EP decomposition.

However, the condition that $R$ is a proper $\ast$-ring is a sufficient but not necessary condition under which each element has at most one weak group inverse (see Example \ref{e1}). This motivates us to study sufficient and necessary conditions under which each element has at most one weak group inverse. In addition, it's natural to consider the conditions under which an element is weak group invertible.

Motivated by the idea of  \cite{MRT} and \cite{ZCP}, we define the $m$-weak group inverse in $R$. Then we give an equivalent definition of the $m$-weak group inverse. In addition, we consider the relation between the $m$-weak group inverse and the $(m+1)$-weak group inverse. Some equivalent characterizations are given when the $m$-weak group inverse is equal to the Drazin inverse.

The rest of this paper is organized as follows. In Section 2, we present some necessary definitions and lemmas. In Section 3, we prove that each element in $R$ has at most one weak group inverse if and only if each idempotent element has a unique weak group inverse. In addition, some equivalent characterizations of the weak group inverse are given. In Section 4, we define the $m$-weak group inverse in a unitary ring with involution and present some properties of $m$-weak group inverses.

\section{\bf Preliminaries}
Recall that an element $a\in R$ is said to be idempotent if $a^{2}=a$. The set of all idempotent elements is denoted by ${\rm idem} (R)$. We will write $aR=\{ax: x\in R\}$ and $Ra=\{xa: x\in R\}$.

\begin{definition}\label{g1} \emph{\cite{KDC}} An element $a\in R$ is left $\ast$-cancellable if $a^{*}ax=a^{*}ay$ implies $ax=ay$, it is right $\ast$-cancellable if $xaa^{*}=yaa^{*}$ implies $xa=ya$, and $\ast$-cancellable if it is both left and right $*$-cancellable. Moreover, $R$ is a proper $\ast$-ring if and only if every element in $R$ is $\ast$-cancellable.
\end{definition}

\begin{definition} \emph{\cite{D}}
Let $a\in R$. If there exist $x\in R$ and $k\in\mathbb{N}^{+}$ such that
$$xa^{k+1}=a^{k},~~~xax=x, ~~~xa=ax,$$
then $x$ is called the Drazin inverse of $a$. It is unique and denoted by $a^{D}$ when the Drazin inverse exists. If $k$ is the smallest positive integer such that above equations hold, then $k$ is the Drazin index of $a$ and denoted by ${\rm i}(a)=k.$
\end{definition}

\begin{definition} \emph{\cite{GC}}
Let $a\in R$. If there exist $x\in R$ and $k\in\mathbb{N}^{+}$ such that
$$xa^{k+1}=a^{k},~~~ax^{2}=x, ~~~(ax)^{*}=ax,$$
then $x$ is called the pseudo core inverse of $a$. It is unique and denoted by $a^{\scriptsize\textcircled{\tiny D}}$ when the pseudo core inverse exists.

\end{definition}

\begin{definition} \emph{\cite{Z}} Let $a\in R$. If there exist $x\in R$ and $k\in\mathbb{N}^{+}$ such that
$$xa^{k+1}=a^{k},~~~ax^{2}=x, ~~~(a^k)^{*}a^{2}x=(a^k)^{*}a,$$
then $x$ is called the weak group inverse of $a$. When the weak group inverse of $a$ is unique, it's denoted by $a^{\scriptsize\textcircled{\tiny W}}$.
\end{definition}

The smallest positive integer $k$ satisfying above equations is called the weak group index of $a$ for $x$ and denoted by ${\rm ind}(a,x)=k$. By the following Lemma \ref{a2}, we can get that $a$ is Drazin invertible if $a$ is weak group invertible. By a similar proof of Proposition 3.4 in \cite{Z}, we have ${\rm ind}(a,x)={\rm i}(a)$. So, if $y$ is another weak group inverse of $a$, we have ${\rm ind}(a,x)={\rm i}(a)={\rm ind}(a,y).$ For this reason, the smallest positive integer $k$ satisfying above equations can be called the weak group index of $a$ and denoted by ${\rm ind}(a)=k$.

\begin{lemma}\label{a2} \emph{\cite{GC}} Let $a\in R$. If there exist $x\in R$ and $k\in \mathbb{N}^{+}$such that
$$ ~xa^{k+1}=a^{k}, ~~~ ax^{2}=x,$$
then we have
\begin{itemize}
\item [{\rm (1)}] $ax=a^{m}x^{m}$ for arbitrary positive integer $m$;
\item [{\rm (2)}] $xax=x$;
\item [{\rm (3)}] $a$ is Drazin invertible, $a^{D}=x^{k+1}a^{k}$, and ${\rm i}(a)\leq k$.
\end{itemize}
\end{lemma}

\begin{corollary}\label{g2}Let $a\in R$. If there exist $x\in R$ and $k\in \mathbb{N}^{+}$such that
$$ ~xa^{k+1}=a^{k}, ~~~ ax^{2}=x,$$
then we have
\begin{itemize}
\item [{\rm (1)}] $aa^{D}x=x$ ;\item [{\rm (2)}] $(ax)^{2}=ax$;\item [{\rm (3)}] $axa^{D}=a^{D}$ .
\end{itemize}
\end{corollary}

\begin{lemma}\label{m1}Let $a,x\in R$, if there exist three positive integers $m,n,k$ such that
$$ ~xa^{k+1}=a^{k}, ~~~ ax^{2}=x,~~~a^{m}x^{n}=x^{n}a^{m}, $$
then $a$ is Drazin invertible and $a^{D}=x$.
\end{lemma}

\begin{proof}By Lemma \ref{a2}, $a$ is Drazin invertible. Then we have
\begin{eqnarray*}
 xa&=&a^{(k+1)mn-1}x^{(k+1)mn}a=a^{m-1}a^{((k+1)n-1)m}x^{(k+1)mn}a\\
 &=&a^{m-1}x^{(k+1)mn}a^{((k+1)n-1)m}a=a^{m-1}x^{(k+1)mn-1}a^{((k+1)n-1)m}\\
 &=&a^{m}x^{(k+1)mn}a^{((k+1)n-1)m}=a^{m}a^{((k+1)n-1)m}x^{(k+1)mn}\\
 &=&a^{(k+1)mn}x^{(k+1)mn}=ax.
\end{eqnarray*}
It is easy to obtain that $a^{D}=x$.
\end{proof}

\begin{lemma}\label{a5}  Let $g, e\in {\rm idem}(R)$ and $gR\subseteq eR$. Then there exists $f \in {\rm idem}(R)$ such that $fR=eR$ and $gf=g$.
\begin{proof}
Taking $f=g+(1-g)e$, we have
\begin{eqnarray*}
 f^{2}&=&g^{2}+g(1-g)e+(1-g)eg+(1-g)e(1-g)e\\
 &=&g+0+0+(1-g)(e-g)e\\
 &=&g+(1-g)e=f.
\end{eqnarray*}
So, $f\in {\rm idem}(R)$. By computation, we can get
\begin{eqnarray*}
 fe&=&ge+(1-g)e^{2}=e,\\
 ef&=&eg+e(1-g)e=g+(e-g)e=g+(1-g)e=f,\\
 gf&=&g^{2}+g(1-g)e=g.
\end{eqnarray*}
Therefore, there exists $f \in {\rm idem}(R)$ such that $fR=eR$ and $gf=g$.
\end{proof}
\end{lemma}

\begin{definition}\label{b1}\emph{\cite{Z}} If $R$ is a proper $\ast$-ring and $a\in R$. Then $a=a_{1}+a_{2}$ is called the group-EP decomposition of $a$, if the following conditions hold:
\begin{itemize}
\item [{\rm (1)}]  $a_{1}\in R^{\#}$;
\item [{\rm (2)}] $a_{2}^{k}=0$, for some $k\in\mathbb{N}^{+}$;
\item [{\rm (3)}] $a_{1}^{*}a_{2}=a_{2}a_{1}=0$.
\end{itemize}
\end{definition}

\section{The existence and uniqueness of weak group inverse}\label{a}

In this section, we give sufficient and necessary conditions under which each element has at most one weak group inverse. At first, we present some equivalent characterizations of the weak group inverse.

\begin{proposition}\label{g5} Let $a\in R$. Then $a$ is weak group invertible if and only if $a$ is Drazin invertible and there exists $f\in {\rm idem}(R)$ such that $fR=aa^{D}R$ and $(a^{D})^{*}af=(a^{D})^{*}a$. In this case, $a^{D}f$ is one of weak group inverses of $a$.
\end{proposition}
\begin{proof}
 Suppose $a$ is weak group invertible with ${\rm ind}(a)=k$ and $x$ is one of weak group inverses of $a$. By Lemma \ref{a2}, $a$ is Drazin invertible and $(ax)^{2}=ax$. Taking $f=ax$, we have
\begin{eqnarray*}
  faa^{D}&=&axaa^{D}=axa^{k+1}(a^{D})^{k+1}=aa^{k}(a^{D})^{k+1}=a^{k+1}(a^{D})^{k+1}=aa^{D},\\
  aa^{D}f&=&aa^{D}ax=aa^{D}a^{k+1}x^{k+1}=aa^{k}x^{k+1}=a^{k+1}x^{k+1}=ax,\\
  (a^{D})^{*}af&=&(a^{D})^{*}a^{2}x=(a^{k}(a^{D})^{k+1})^{*}a^{2}x=((a^{D})^{k+1})^{*}(a^{k})^{*}a^{2}x\\
  &=&((a^{D})^{k+1})^{*}(a^{k})^{*}a=(a^{k}(a^{D})^{k+1})^{*}a=(a^{D})^{*}a.
\end{eqnarray*}
Therefore $fR=aa^{D}R$ and $(a^{D})^{*}af=(a^{D})^{*}a$.

Conversely, suppose $a$ is Drazin invertible with ${\rm i}(a)=k.$ Taking $y=a^{D}f$, we have
\begin{eqnarray*}
  ya^{k+1}&=&a^{D}fa^{k+1}\\
  &=&a^{D}faa^{D}a^{k+1}\\
  &=&a^{D}aa^{D}a^{k+1}\\
  &=&a^{k},\\
  ay^{2}&=&aa^{D}fa^{D}f=fa^{D}f\\
  &=&fa^{D}(aa^{D}f)=(faa^{D})a^{D}f\\
  &=&aa^{D}a^{D}f=y
\end{eqnarray*}
and
\begin{eqnarray*}
  (a^{k})^{*}a^{2}y&=&(a^{k})^{*}a^{2}a^{D}f\\
  &=&(a^{D}a^{k+1})^{*}af\\
  &=&(a^{k+1})^{*}(a^{D})^{*}af\\
  &=&(a^{k+1})^{*}(a^{D})^{*}a\\
  &=& (a^{k})^{*}a.
\end{eqnarray*}
Therefore $a$ is weak group invertible.
\end{proof}

\begin{corollary}\label{t1} Let $a\in R$. Then $a$ is weak group invertible if and only if $a$ is Drazin invertible and there exists $g\in {\rm idem}(R)\cap aa^{D}R$ such that $(a^{D})^{*}ag=(a^{D})^{*}a$.
\end{corollary}
\begin{proof}
Suppose $a$ is weak group invertible. By Proposition \ref{g5}, it's obvious that there exists $g\in {\rm idem}(R)\cap aa^{D}R$ such that $(a^{D})^{*}ag=(a^{D})^{*}a$.

 Conversely, suppose there exists $g\in {\rm idem}(R)\cap aa^{D}R$ such that $(a^{D})^{*}ag=(a^{D})^{*}a$. Set $e=aa^{D}$. By Lemma \ref{a5}, there exists $f\in {\rm idem}(R)$ such that  $fR=eR$ and $gf=g$. Then we have $(a^{D})^{*}af=((a^{D})^{*}ag)f=(a^{D})^{*}ag=(a^{D})^{*}a$. By Proposition \ref{g5}, $a$ is weak group invertible.
\end{proof}

\begin{corollary}\label{g6} Let $a\in R$. If $a$ is Drazin invertible, then $a$ has a unique weak group inverse if and only if there exists a unique $f\in {\rm idem}(R)$ such that $fR=aa^{D}R$ and $(a^{D})^{*}af=(a^{D})^{*}a$.
\end{corollary}
\begin{proof}It's obvious from the proof of Proposition \ref{g5}.
\end{proof}

The following proposition gives the equivalent condition which ensures that an element in $R$ has a unique weak group inverse.
\begin{proposition}\label{g7}Let $a\in R$. If $a$ is weak group invertible with ${\rm ind}(a)=k$, then the following conditions are equivalent:
\begin{itemize}
\item [{\rm (1)}] $a$ has a unique weak group inverse;
\item [{\rm (2)}] For any element $x\in aa^{D}R(1-aa^{D})$, if $(aa^{D})^{*}x=0$, then $x=0$.
\end{itemize}
\end{proposition}
\begin{proof}
${\rm (1)}\Rightarrow{\rm (2)}$: Suppose $y$ is the unique weak group inverse of $a$, $x\in aa^{D}R(1-aa^{D})$ and $(aa^{D})^{*}x=0$. By taking $z=y+(a^{D})^{2}x$, we claim that $z$ is also a weak group inverse of $a$. Since $x\in aa^{D}R(1-aa^{D})$, $aa^{D}x=x$ and $xa^{D}=x(1-aa^{D})a^{D}=0$. By Corollary \ref{g2}, we have $aa^{D}y=y$ and $aya^{D}=a^{D}$. Then we can get
\begin{eqnarray*}
  za^{k+1}&=&(y+(a^{D})^{2}x)a^{k+1}\\
  &=&a^{k}+(a^{D})^{2}xaa^{D}a^{k+1}\\
  &=&a^{k},\\
  az^{2}&=&a(y^{2}+y(a^{D})^{2}x+(a^{D})^{2}xy+((a^{D})^{2}x)^{2})\\
  &=&a(y^{2}+y(a^{D})^{2}x+(a^{D})^{2}xaa^{D}y)\\
  &=&a(y^{2}+y(a^{D})^{2}x)\\
  &=&y+ay(a^{D})^{2}x\\
  &=&y+(a^{D})^{2}x=z
\end{eqnarray*}
and
\begin{eqnarray*}
(a^{k})^{*}a^{2}z&=&(a^{k})^{*}a^{2}(y+(a^{D})^{2}x)\\
&=&(a^{k})^{*}a+(a^{k})^{*}a^{2}(a^{D})^{2}x\\
&=&(a^{k})^{*}a+(a^{k})^{*}x\\
&=&(a^{k})^{*}a+(aa^{D}a^{k})^{*}x\\
&=&(a^{k})^{*}a+(a^{k})^{*}(aa^{D})^{*}x\\
&=&(a^{k})^{*}a.
\end{eqnarray*}
Since $a$ has a unique weak group inverse, $y=z$, that is, $(a^{D})^{2}x=0$. Then $x=(aa^{D})^{2}x=a^{2}(a^{D})^{2}x=0$.\\
${\rm (2)}\Rightarrow{\rm (1)}$: Assume that $y$ and $z$ satisfy the conditions of definition of the weak group inverse. By Corollary \ref{g2}, we have $aa^{D}y=y$ and $aya^{D}=a^{D}$ and so is $z$. Taking $x=a^{2}(y-z)$, we have $aa^{D}x=x$ and $xaa^{D}=0$. So $x\in aa^{D}R(1-aa^{D})$. By computation, we can get
\begin{eqnarray*}
(aa^{D})^{*}x&=&(a^{k}(a^{D})^{k})^{*}a^{2}(y-z)\\
&=&((a^{D})^{k})^{*}(a^{k})^{*}a^{2}(y-z)\\
&=&((a^{D})^{k})^{*}(a^{k})^{*}a-((a^{D})^{k})^{*}(a^{k})^{*}a=0
\end{eqnarray*}
Since $x\in aa^{D}R(1-aa^{D})$ and $(aa^{D})^{*}x=0$, $x=0$, that is, $a^{2}y=a^{2}z$. Then we have $y=(aa^{D})^{2}y=(a^{D})^{2}a^{2}y=(a^{D})^{2}a^{2}z=z$. Therefore $a$ has a unique weak group inverse.
\end{proof}

It is obvious that an idempotent element in $R$ is weak group invertible. By Corollary \ref{g6} and Proposition \ref{g7}, we can obtain the equivalent conditions which guarantee that each element in $R$ has at most one weak group inverse.
\begin{theorem}\label{g8}The following conditions are equivalent:
\begin{itemize}
\item [{\rm (1)}] Each element in $R$ has at most one weak group inverse;
\item [{\rm (2)}] Each idempotent element in $R$ has a unique weak group inverse;
\item [{\rm (3)}] For any $e,f\in  {\rm idem}(R)$, if $eR=fR$ and $e^{*}f=e^{*}e$, then $e=f$;
\item [{\rm (4)}] For any $e\in {\rm idem}(R)$ and $x\in R$, if $e^{*}ex(1-e)=0$, then $ex(1-e)=0$.
\end{itemize}
\end{theorem}
\begin{proof}
${\rm (1)}\Rightarrow{\rm (2)}$: Obvious.\\
${\rm (2)}\Rightarrow{\rm (3)}$: Suppose that $e,f\in {\rm idem}(R)$, $eR=fR$ and $e^{*}f=e^{*}e$. By Corollary\ref{g6}, we can easily get $e=f$. \\
${\rm (3)}\Rightarrow{\rm (4)}$: Suppose $e\in {\rm idem}(R)$, $x\in R$ and $e^{*}ex(1-e)=0$. Taking $f=e-ex(1-e)$, we have $f^{2}=f$, $ef=f$, $fe=e$ and $e^{*}f=e^{*}e$. So, $f\in {\rm idem}(R)$ and $eR=fR$. By the condition ${\rm (3)}$, we have $e=f$, that is, $ex(1-e)=0$. \\
${\rm (4)}\Rightarrow{\rm (1)}$: Suppose $a\in R$ is weak group invertible. By the condition ${\rm (4)}$, we can get that if $(aa^{D})^{*}x=0$, then $x=0$ for any $x\in aa^{D}R(1-aa^{D})$. By Proposition \ref{g7}, $a$ has a unique weak group inverse.
\end{proof}

Now, we give an example to show the application of Theorem \ref{g8}.

\begin{example} Let $R=M_{2}(\mathbb{C})$. Taking the involution is transpose. Set $A=\left(\begin{matrix}
1&0\\
i&0
\end{matrix}
\right)$ and $B=\left(\begin{matrix}
0&1\\
0&0
\end{matrix}
\right)$. Since $A^2=A$, $A^{*}AB(E-A)=0$ and $AB(E-A)\neq0$, $R$ is not a weak proper $\ast$-ring by Theorem \ref{g8}. In fact, it is easy to verify that $X_{1}=\left(\begin{matrix}
1&0\\
i&0
\end{matrix}
\right)$ and $X_{2}=\left(\begin{matrix}
0&-i\\
0&1
\end{matrix} 
\right)$ are two weak group inverses of $A$.
\end{example}

If each element in $R$ has at most one weak group inverse, $R$ is called weak proper $\ast$-ring.

\begin{corollary}\label{r2} If each idempotent element in $R$ is left $\ast$-cancellable, then $R$ is
a weak proper $\ast$-ring.
\end{corollary}

The following example shows that an idempotent element in a weak proper $\ast$-ring may not be left $\ast$-cancellable. In addition, this example shows that $R$ is a weak proper $\ast$-ring, but $R$ is not a proper $\ast$-ring.

\begin{example}\label{e1}Let $G=\{e,a,b,c\}$ be a non-cyclic group. Let $R=\mathbb {Z}_{3}G$ with an involution, where $(x_{1}e+x_{2}a+x_{3}b+x_{4}c)^{*}=x_{1}e+x_{2}b+x_{3}a+x_{4}c$. Since $R$ is a commutative ring, $R$ is a weak proper $\ast$-ring. Since $(2e+a)^{*}(2e+a)(2e+2b)=0$ and $(2e+a)(2e+2b)\ne 0$, $(2e+a)^{2}=2e+a$ is not left $\ast$-cancellable. Also, $R$ is not a proper $\ast$-ring.
\end{example}

\begin{example} Let $R=\mathbb{R} \times \mathbb{R}$ with an involution, where $(a,b)^*=(b,a)$ for any $a, b\in \mathbb{R}$. Since $R$ is a commutative ring, $R$ is a weak proper $\ast$-ring. Set $e=(1,0)$. Since $e^*e=0$ and $e\neq0$, $e=e^2$ is not left $\ast$-cancellable. That is, an idempotent element in a weak proper $\ast$-ring may not be left $\ast$-cancellable.
\end{example}

The following example shows that when each idempotent element in $R$ is left $\ast$-cancellable, $R$ may not be a proper $\ast$-ring.

\begin{example}Let $R=\mathbb{Z}_{4}$ with an involution, where $a^*=a$ for any $a\in R$. Since $2^{*}\times2=0$ and $2\neq0$, $R$ is not a proper $\ast$-ring. It is obvious that each idempotent element in $R$ is left $\ast$-cancellable.
\end{example}

Now, we can give an existence criteria for weak group inverses in a ring where each idempotent element is left $\ast$-cancellable.

\begin{proposition}\label{g9}If each idempotent element in $R$ is left $\ast$-cancellable and $a\in R$ is Drazin invertible with ${\rm i}(a)=k$, then $a$ is weak group invertible if and only if there exists $x\in R$ such that $(a^{D})^{*}a=(a^{D})^{*}a^{D}x$. In this case, $a^{\scriptsize\textcircled{\tiny W}}=(a^{D})^{3}x$.
\end{proposition}
\begin {proof} Suppose $a$ is weak group invertible. Taking $x=a^{3}a^{\scriptsize\textcircled{\tiny W}}$, by Corollary \ref{g2} we have
\begin{eqnarray*}
(a^{D})^{*}a^{D}x&=&(a^{D})^{*}a^{D}a^{3}a^{\scriptsize\textcircled{\tiny W}}\\
&=&(a^{D})^{*}a^{2}a^{\scriptsize\textcircled{\tiny W}}\\
&=&(a^{k}(a^{D})^{k+1})^{*}a^{2}a^{\scriptsize\textcircled{\tiny W}}\\
&=&((a^{D})^{k+1})^{*}(a^{k})^{*}a^{2}a^{\scriptsize\textcircled{\tiny W}}\\
&=&((a^{D})^{k+1})^{*}(a^{k})^{*}a\\
&=&(a^{D})^{*}a.
\end{eqnarray*}
Therefore there exists $x\in R$ such that $(a^{D})^{*}a=(a^{D})^{*}a^{D}x$.

Conversely, if there exists $x\in R$ such that $(a^{D})^{*}a=(a^{D})^{*}a^{D}x$. Since $(a^{D})^{*}a=(a^{D})^{*}a^{D}x$, $(aa^{D})^{*}aa^{D}=(aa^{D})^{*}aa^{D}(a^{D}xa^{D})$. Because each idempotent element in $R$ is left $\ast$-cancellable, $aa^{D}=a^{D}xa^{D}$. Taking $y=(a^{D})^{3}x$, we have
\begin{eqnarray*}
ya^{k+1}&=&(a^{D})^{3}xa^{k+1}\\
&=&(a^{D})^{3}xa^{D}a^{k+2}\\
&=&(a^{D})^{2}aa^{D}a^{k+2}\\
&=&a^{k},\\
ay^{2}&=&a(a^{D})^{3}x(a^{D})^{3}x\\
&=&a(a^{D})^{2}aa^{D}(a^{D})^{2}x\\
&=&(a^{D})^{3}x=y
\end{eqnarray*}
and
\begin{eqnarray*}
(a^{k})^{*}a^{2}y&=&(a^{k})^{*}a^{D}x\\
&=&(a^{D}a^{k+1})^{*}a^{D}x\\
&=&(a^{k+1})^{*}(a^{D})^{*}a^{D}x\\
&=&(a^{k+1})^{*}(a^{D})^{*}a\\
&=&(a^{k})^{*}a.
\end{eqnarray*}
Therefore $a$ is weak group invertible.
\end{proof}

Zhou et al.\cite{Z} proved that an element in a proper $\ast$-ring is weak group invertible if and only if this element has group-EP decomposition. We now generalize the group-EP decomposition to weak proper $\ast$-rings.
\begin{definition} If $R$ is a weak proper $\ast$-ring. Let $a\in R$. Then $a=a_{1}+a_{2}$ is called the group-EP decomposition of $a$, if the following conditions hold:
\begin{itemize}
\item [{\rm (1)}]  $a_{1}\in R^{\#}$;
\item [{\rm (2)}] $a_{2}^{k}=0$, for some $k\in\mathbb{N}^{+}$;
\item [{\rm (3)}] $a_{1}^{*}a_{2}=a_{2}a_{1}=0$.
\end{itemize}
\end{definition}

\begin{theorem}\label{b2} If $R$ is a weak proper $\ast$-ring. Let $a\in R$. Then $a$ is weak group invertible if and only if $a$ has group-EP decomposition.
\end{theorem}

\begin{proof}  Suppose $a^{\scriptsize\textcircled{\tiny W}}$ exists with ${\rm ind}(a)=k$. Set $a_{1}=a^{2}a^{\scriptsize\textcircled{\tiny W}}$ and $a_{2}=a-a^{2}a^{\scriptsize\textcircled{\tiny W}}$. By a similar proof of Theorem 4.2 in \cite{Z}, we have $a_2a_1=a_1^{*}a_2=0$ and $a_1$ is group invertible. Now, we only need to prove $a_2^{k}=0.$ By induction, we have $a_{2}^{k}=(a-a^{2}a^{\scriptsize\textcircled{\tiny W}})^{k}=a^{k}-a^{2}a^{\scriptsize\textcircled{\tiny W}}a^{k-1}$. Then we can get $aa^{D}a_{2}^{k}=a_{2}^{k}$ and $a_{2}^{k}(1-aa^{D})=a_{2}^{k}$ by Corollary \ref{g2}. So $a_{2}^{k}\in aa^{D}R(1-aa^{D})$. Then we have
\begin{eqnarray*}
(aa^{D})^{*}a_{2}^{k}&=&(aa^{D})^{*}(a^{k}-a^{2}a^{\scriptsize\textcircled{\tiny W}}a^{k-1})\\
&=&(a^{k}(a^{D})^{k})^{*}(a^{k}-a^{2}a^{\scriptsize\textcircled{\tiny W}}a^{k-1})\\
&=&((a^{D})^{k})^{*}((a^{k})^{*}a^{k}-(a^{k})^{*}a^{2}a^{\scriptsize\textcircled{\tiny W}}a^{k-1})\\
&=&0
\end{eqnarray*}
By Theorem \ref{g8}, we can get $a_2^{k}=0$.

Conversely, it's similar to the proof in \cite{Z}.
\end{proof}

\begin{theorem} If $R$ is a weak proper $\ast$-ring, then the group-EP decomposition in $R$ is unique.
\end{theorem}

\begin{proof}  It is similar to the proof in \cite{Z}.
\end{proof}

\section{$m$-weak group inverse}\label{b}
In this section, we stipulate that $a^{0}=1$ for any $a\in R$. Then we present the definition of the $m$-weak group inverse.
\begin{definition}\label{r3}Let $a\in R$ and $m\in \mathbb N.$ If there exist $x\in R$ and $k\in \mathbb{N}^{+}$ such that
\begin{equation*}
 {\rm (i)}~ xa^{k+1}=a^{k},~~~{\rm (ii)} ~ax^{2}=x,~~~{\rm (iii)}~ (a^{k})^{*}a^{m+1}x=(a^{k})^{*}a^{m},
\end{equation*}
then $x$ is called the $m$-weak group inverse of $a$. When the $m$-weak group inverse of $a$ is unique, it's denoted by $a^{\scriptsize\textcircled{\tiny W}_m}$.\\
\end{definition}

By Lemma \ref{a2}, if $a$ is $m$-weak group invertible, then $a$ is Drazin invertible. If $a$ is Drazin invertible with ${\rm i}(a)=k$, then $a$ is $(k+s)$-weak group invertible for any $s\in \mathbb{N}$ and $a^{D}$ is one of $(m+s)$-weak group inverses of $a$. When $a$ is $m$-weak group invertible, the $m$-weak group inverse of $a$ may not be unique. We can define the $m$-weak group index which is similar to the weak group index. The smallest positive integer $k$ satisfying Eqs. ${\rm (i)}$-${\rm (iii)}$ can be called the $m$-weak group index of $a$ and denoted by ${\rm ind_m}(a)=k$. In addition, we have ${\rm ind_m}(a)={\rm i}(a)$ if $a$ is $m$-weak group invertible. Then, we can give another equivalent definition of the $m$-weak group inverse in the following.\\

\begin{proposition}\label{t2} Let $a\in R$ and $m\in \mathbb N$. Then $a$ is $m$-weak group invertible if and only if there exist $x\in R$ and $k\in \mathbb{N}^{+}$ such that
\begin{equation*}
 {\rm (i)}~ xa^{k+1}=a^{k},~~~{\rm (ii)} ~ax^{2}=x,~~~{\rm (iv)}~ ((a^{m})^{*}a^{m+1}x)^{*}=(a^{m})^{*}a^{m+1}x.
 \end{equation*}
\end{proposition}
\begin{proof}

If there exist $x\in R$ and $k\in \mathbb{N}^{+}$ satisfying Eqs. ${\rm (i)}$, ${\rm (ii)}$, ${\rm (iii)}$, then we have
\begin{eqnarray*}
((a^{m})^{*}a^{m+1}x)^{*}&=&((a^{m})^{*}a^{m+1+k}x^{k+1})^{*}\\
&=&(a^{m+1+k}x^{k+1})^{*}a^{m}\\
&=&(a^{m+1+k}x^{k+1})^{*}a^{m+1}x\\
&=&(a^{m+1}x)^{*}a^{m+1}x
\end{eqnarray*}
Therefore $((a^{m})^{*}a^{m+1}x)^{*}=(a^{m})^{*}a^{m+1}x$.

Conversely, if there exist $x\in R$ and $k\in \mathbb{N}^{+}$ satisfying Eqs. ${\rm (i)}$, ${\rm (ii)}$, ${\rm (iv)}$, then we have
\begin{eqnarray*}
(a^{k})^{*}a^{m+1}x&=&(xa^{k+1})^{*}a^{m+1}x\\
&=&(a^{m}x^{m+1}a^{k+1})^{*}a^{m+1}x\\
&=&(x^{m+1}a^{k+1})^{*}(a^{m})^{*}a^{m+1}x\\
&=&(x^{m+1}a^{k+1})^{*}((a^{m})^{*}a^{m+1}x)^{*}\\
&=&((a^{m})^{*}a^{m+1}x^{m+2}a^{k+1})^{*}\\
&=&((a^{m})^{*}xa^{k+1})^{*}\\
&=&(a^{k})^{*}a^{m}.
\end{eqnarray*}
\end{proof}

\begin{corollary}\label{m2} Let $a\in R$. Then $a$ is pseudo core invertible if and only if $a$ is 0-weak group invertible. $a$ has at most one $0$-weak group inverse. In this case, $a^{\scriptsize\textcircled{\tiny D}}=a^{\scriptsize\textcircled{\tiny W}_0}$.
\end{corollary}

\begin{corollary}Let $a\in R$. Then $a$ is weak group invertible if and only if $a$ is $1$-weak group invertible.
\end{corollary}

According to Proposition \ref{t2}, we present another equivalent definition of the weak group inverse.

\begin{corollary}Let $a\in R$. Then $a$ is weak group invertible if and only if there exist $x\in R$ and $k\in \mathbb{N}^{+}$ such that
\begin{equation*}
 {\rm (i)}~ xa^{k+1}=a^{k},~~~{\rm (ii)} ~ax^{2}=x,~~~{\rm (iv)}~ (a^{*}a^{2}x)^{*}=a^{*}a^{2}x,
\end{equation*}.
\end{corollary}

\begin{proposition}\label{m4} Let $a\in R$ and $m\in \mathbb N^{+}$. Then $a$ is $m$-weak group invertible if and only if $a^{m}$ is weak group invertible.
\end{proposition}
\begin{proof}Suppose there exist $x\in R$ and $k\in \mathbb{N}^{+}$ satisfying Eqs. ${\rm (i)}$, ${\rm (ii)}$, ${\rm (iii)}$. Taking $y=x^{m}$, we have
\begin{eqnarray*}
y(a^{m})^{k+1}&=&x^{m}(a^{m})^{k+1}=(a^{m})^{k},\\
a^{m}y^{2}&=&a^{m}x^{2m}=x^{m}=y
\end{eqnarray*}
and\\
\begin{eqnarray*}
((a^{m})^{k})^{*}(a^{m})^{2}y&=&((a^{m})^{k})^{*}a^{2m}x^{m}\\
&=&((a^{m})^{k})^{*}a^{m+1}x\\
&=&((a^{m})^{k})^{*}a^{m}.
\end{eqnarray*}
So, $a^{m}$ is weak group invertible.

Conversely, suppose $y$ is one of weak group inverses of $a^{m}$ with ${\rm ind}(a)=k$. Taking $x=a^{m-1}y$, we have
\begin{eqnarray*}
xa^{m(k+1)+1}&=&a^{m-1}ya^{m(k+1)+1}\\
&=&a^{m-1}a^{mk+1}=a^{m(k+1)},\\
ax^{2}&=&a^{2m-1}y^{2}=a^{m-1}y=x
\end{eqnarray*}
and\\
\begin{eqnarray*}
(a^{m(k+1)})^{*}a^{m+1}x&=&(a^{m(k+1)})^{*}a^{2m}y\\
&=&((a^{m(k+1)})^{*}a^{m}.
\end{eqnarray*}
So, $a$ is $m$-weak group invertible by Definition \ref{r3}.
\end{proof}
The following example and proposition show the relation between the $m$-weak group inverse and the $(m+1)$-weak group inverse.
\begin{example} Let $R=M_{4}(\mathbb{Z})$, where $\mathbb{Z}$ denotes the set of all integers. Taking the involution is transpose. It is obvious that $R$ is a proper $\ast$-ring. Set $A=\left(\begin{matrix}
1&0&0&0\\
1&0&1&0\\
0&0&0&2\\
0&0&0&0
\end{matrix}
\right)$. By computation, we have $A^{2}=\left(\begin{matrix}
1&0&0&0\\
1&0&0&2\\
0&0&0&0\\
0&0&0&0
\end{matrix}
\right)$ and $A^{D}=(A^{2})^{D}=\left(\begin{matrix}
1&0&0&0\\
1&0&0&0\\
0&0&0&0\\
0&0&0&0
\end{matrix}
\right)$ with ${\rm i}(A)=3$ and ${\rm i}(A^{2})=2$. Since $(A^{D})^{\mathrm{T}}A\not\in (A^{D})^{\mathrm{T}}A^{D}R$ and $((A^{2})^{D})^{\mathrm{T}}A^{2}\in ((A^{2})^{D})^{\mathrm{T}}(A^{2})^{D}R$, we can get that $A^{2}$ is weak group invertible but $A$ is not weak group invertible by Proposition $\ref{g9}$. According to Proposition $\ref{m4}$, we have $A$ is $2$-weak group invertible but is not $1$-weak group invertible.
\end{example}

\begin{proposition}Let $a\in R$ and $m\in \mathbb N$. If $a$ is $m$-weak group invertible and $x\in R$ is one of $m$-weak group inverses of $a$, then $a$ is $(m+1)$-weak group invertible and $x^{2}a$ is one of $(m+1)$-weak group inverses of $a$.
\end{proposition}
\begin{proof}Suppose $x$ is one of $m$-weak group inverses of $a$ with ${\rm ind_m}(a)=k$. Taking $y=x^{2}a$, we have
\begin{eqnarray*}
ya^{k+1}&=&x^{2}a^{k+2}=a^{k},\\
ay^{2}&=&ax^{2}ax^{2}a=x^{2}a=y
\end{eqnarray*}
and\
\begin{eqnarray*}
(a^{k})^{*}a^{m+2}y&=&(a^{k})^{*}a^{m+2}x^{2}a\\
&=&(a^{k})^{*}a^{m+1}xa\\
&=&(a^{k})^{*}a^{m+1}.
\end{eqnarray*}
So, $a$ is $(m+1)$-weak group invertible and $x^{2}a$ is one of $(m+1)$-weak group inverses of $a$.
\end{proof}

\begin{corollary}\label{t3}Let $a\in R$ and $m,s\in \mathbb N$. If $a$ is $m$-weak group invertible and $x\in R$ is one of $m$-weak group inverses of $a$, then $a$ is $(m+s)$-weak group invertible and $x^{s+1}a^{s}$ is one of $(m+s)$-weak group inverses of $a$.
\end{corollary}

The following proposition gives the equivalent condition which guarantees that an element in $R$ has a unique $m$-weak group inverse.

\begin{proposition}Let $a\in R$ and integers $m, n$ satisfy $0<m<n$. If $a$ is $m$-weak group invertible, then the following conditions are equivalent:
\begin{itemize}
\item [{\rm (1)}] $a$ has a unique $m$-weak group inverse;
\item [{\rm (2)}] For any element $x\in aa^{D}R(1-aa^{D})$, if $(aa^{D})^{*}x=0$, then $x=0$;
\item[{\rm (3)}] $a$ has a unique $n$-weak group inverse.
\end{itemize}
\end{proposition}
\begin{proof}It is similar to the proof in Proposition \ref{g7}.
\end{proof}

\begin{corollary} Let $a\in R$ and $m\in \mathbb {N^{+}}$. If $a$ is pseudo core invertible, then $a$ has a unique $m$-weak group inverse.
\end{corollary}

\begin{proposition}Let $m\in \mathbb N^{+}$. Then each element in $R$ has at most one $m$-weak group inverse if and only if $R$ is a weak proper $\ast$-ring.
\end{proposition}
\begin{proof}It is similar to the proof in Theorem \ref{g8}.
\end{proof}

Some equivalent characterizations are presented in the following theorem when the $m$-weak group inverse is equal to the Drazin inverse.
\begin{theorem}\label{t4}Let $a\in R$ and integers $m, n$ satisfy $0\leq m<n$. If $a$ has a unique $m$-weak group inverse, then the following conditions are equivalent:
\begin{itemize}
\item[{\rm (1)}] $a^{\scriptsize\textcircled{\tiny W}_m}a=aa^{\scriptsize\textcircled{\tiny W}_m}$;
\item[{\rm (2)}] $(a^{n})^{\scriptsize\textcircled{\tiny W}}=(a^{\scriptsize\textcircled{\tiny W}_m}) ^{n}$;
\item[{\rm (3)}] $a^{\scriptsize\textcircled{\tiny W}_m}=a^{\scriptsize\textcircled{\tiny W}_n}$;
\item[{\rm (4)}] $a^{\scriptsize\textcircled{\tiny W}_m}=a^{D}$.
\end{itemize}
\end{theorem}
\begin{proof}
${\rm (1)}\Rightarrow{\rm (2)}$: Since $a^{\scriptsize\textcircled{\tiny W}_m}a=aa^{\scriptsize\textcircled{\tiny W}_m}$, by Corollary \ref{t3} we have
\begin{eqnarray*}
a^{\scriptsize\textcircled{\tiny W}_n}=(a^{\scriptsize\textcircled{\tiny W}_m})^{n-m+1}a^{n-m}=a^{n-m}(a^{\scriptsize\textcircled{\tiny W}_m})^{n-m+1}=a^{\scriptsize\textcircled{\tiny W}_m}.
\end{eqnarray*}
Then by Proposition \ref{m4}, we have $(a^{n})^{\scriptsize\textcircled{\tiny W}}=(a^{\scriptsize\textcircled{\tiny W}_n}) ^{n}=(a^{\scriptsize\textcircled{\tiny W}_m}) ^{n}$.\\
${\rm (2)}\Rightarrow{\rm (3)}$: By Proposition \ref{m4}, we have $(a^{n})^{\scriptsize\textcircled{\tiny W}}=(a^{\scriptsize\textcircled{\tiny W}_n}) ^{n}$. Then we have $$a^{\scriptsize\textcircled{\tiny W}_m}=a^{n-1}(a^{\scriptsize\textcircled{\tiny W}_m}) ^{n}=a^{n-1}(a^{n})^{\scriptsize\textcircled{\tiny W}}=a^{n-1}(a^{\scriptsize\textcircled{\tiny W}_n}) ^{n}=a^{\scriptsize\textcircled{\tiny W}_n}.$$
${\rm (3)}\Rightarrow{\rm (4)}$: By Corollary \ref{t3}, we have $a^{\scriptsize\textcircled{\tiny W}_n}=(a^{\scriptsize\textcircled{\tiny W}_m})^{n-m+1}a^{n-m}$. Then we can get
\begin{eqnarray*}
a^{n-m}(a^{\scriptsize\textcircled{\tiny W}_m})^{n-m+1}=a^{\scriptsize\textcircled{\tiny W}_m}=a^{\scriptsize\textcircled{\tiny W}_n}=(a^{\scriptsize\textcircled{\tiny W}_m})^{n-m+1}a^{n-m}.
\end{eqnarray*}
By Lemma \ref{m1}, it is easy to obtain that $a^{\scriptsize\textcircled{\tiny W}_m}=a^{D}$.\\
${\rm (4)}\Rightarrow{\rm (1)}$: Obvious.
\end{proof}

\begin{corollary}Let $a\in R$ and $m, n\in \mathbb{N^{+}}$. If $a$ is pseudo core invertible, then the following conditions are equivalent:
\begin{itemize}
\item[{\rm (1)}] $a^{\scriptsize\textcircled{\tiny D}}a=aa^{\scriptsize\textcircled{\tiny D}}$;
\item[{\rm (2)}] $(a^{m})^{\scriptsize\textcircled{\tiny W}}=(a^{\scriptsize\textcircled{\tiny D}}) ^{m}$;
\item[{\rm (3)}] $a^{\scriptsize\textcircled{\tiny D}}=a^{\scriptsize\textcircled{\tiny W}_m}$;
\item[{\rm (4)}] $a^{\scriptsize\textcircled{\tiny D}}=a^{D}$;
\item[{\rm (5)}] $(a^{\scriptsize\textcircled{\tiny D}})^{m}a^{n}=a^{n}(a^{\scriptsize\textcircled{\tiny D}})^{m}$.
\end{itemize}
\begin{proof}It is obvious by Lemma \ref{m1} and Theorem \ref{t4}.
\end{proof}
\end{corollary}

\begin{corollary}Let $a\in R$, $m, n, k\in \mathbb{N^{+}}$ and $k\ge 2$. If  $a$ has a unique weak group inverse, then the following conditions are equivalent:
\begin{itemize}
\item[{\rm (1)}] $a^{\scriptsize\textcircled{\tiny W}}a=aa^{\scriptsize\textcircled{\tiny W}}$;
\item[{\rm (2)}] $(a^{k})^{\scriptsize\textcircled{\tiny W}}=(a^{\scriptsize\textcircled{\tiny W}})^{k}$;
\item[{\rm (3)}] $a^{\scriptsize\textcircled{\tiny W}}=a^{\scriptsize\textcircled{\tiny W}_k}$;
\item[{\rm (4)}] $a^{\scriptsize\textcircled{\tiny W}}=a^{D}$;
\item[{\rm (5)}] $(a^{\scriptsize\textcircled{\tiny W}})^{m}a^{n}=a^{n}(a^{\scriptsize\textcircled{\tiny W}})^{m}$.
\end{itemize}
\end{corollary}

\bigskip

\centerline {\bf ACKNOWLEDGMENTS}
This research is supported by the National Natural Science Foundation of China (No. 11771076, 11871145), the Fundamental Research Funds for the Central Universities, the Postgraduate Research and Practice Innovation Program of Jiangsu Province (No. KYCX18$_{-}$0053).

\end{document}